\documentclass{amsart}

\usepackage[utf8]{inputenc}
\usepackage{amsmath,amssymb,amsthm,units, stmaryrd,stackrel,relsize,bm}
\usepackage[textsize=footnotesize,color=yellow, bordercolor=white]{todonotes}

\newtheorem{theorem}{Theorem}[section]
\newtheorem{conjecture}[theorem]{Conjecture}
\newtheorem{lemma}[theorem]{Lemma}

\newtheorem{claim}[theorem]{Claim}
\newtheorem{corollary}[theorem]{Corollary}

\theoremstyle{definition}
\newtheorem{definition}[theorem]{Definition}

\theoremstyle{remark}
\newtheorem{remark}[theorem]{Remark}

\usepackage[
    colorlinks=true, 
    citecolor=red,
    linkcolor=blue,
    urlcolor=blue]{hyperref}
\newcommand\ad{{\mathsf{AD}}}

\newcommand\Ord{{\mathsf{Ord}}}

\newcommand\coll{\text{Coll}}

\newcommand{\PI}{\boldsymbol\Pi}
\newcommand{\SIGMA}{\boldsymbol\Sigma}

\makeatletter
\def\Ddots{\mathinner{\mkern1mu\raise\p@
\vbox{\kern7\p@\hbox{.}}\mkern2mu
\raise4\p@\hbox{.}\mkern2mu\raise7\p@\hbox{.}\mkern1mu}}
\makeatother
\newcommand{\langpm}{\mathcal{L}_{\mathrm{pm}}(\{\dot x_i : i \in \omega\})} 
\newcommand{\lpm}{\mathcal{L}_{\mathrm{pm}}}

\newcommand{\cM}{\mathcal{M}}
\newcommand{\cN}{\mathcal{N}}

\begin{document}
\title{Projective Games on the Reals}


\author{Juan P. Aguilera}
\address{Juan P. Aguilera, Department of Mathematics, Ghent University. Krijgslaan 281-S8, 9000 Ghent, Belgium.} 
\address{Institut f\"ur diskrete Mathematik und
  Geometrie, Technische Universit\"at Wien. Wiedner Hauptstrasse 8-10,
  1040 Wien, Austria.} 
\email{aguilera@logic.at}
\author{Sandra M\"uller} 
\address{Sandra M\"uller, Kurt G\"odel
  Research Center, Institut f\"ur Mathematik, UZA 1, Universit\"at
  Wien. Augasse 2-6, 1090 Wien, Austria.}
\email{mueller.sandra@univie.ac.at} 
\thanks{The first-listed author was partially supported by FWO grant 3E017319 and FWF grants P 31063 and P 31955; the second-listed author, formerly known as Sandra Uhlenbrock, was
  partially supported by FWF grant number P 28157.} 
%\begin{document}

\begin{abstract}
Let $M^\sharp_n(\mathbb{R})$ denote the minimal active iterable extender model which has $n$ Woodin cardinals and contains all reals, if it exists, in which case we denote by $M_n(\mathbb{R})$ the class-sized model obtained by iterating the topmost measure of $M_n(\mathbb{R})$ class-many times. We characterize the sets of reals which are $\Sigma_1$-definable from $\mathbb{R}$ over $M_n(\mathbb{R})$, under the assumption that projective games on reals are determined:
\begin{enumerate}
\item for even $n$, $\Sigma_1^{M_n(\mathbb{R})} = \Game^\mathbb{R}\Pi^1_{n+1}$;
\item for odd $n$, $\Sigma_1^{M_n(\mathbb{R})} = \Game^\mathbb{R}\Sigma^1_{n+1}$.
\end{enumerate}
This generalizes a theorem of Martin and Steel for $L(\mathbb{R})$, i.e.,  the case $n=0$.
As consequences of the proof, we see that determinacy of all projective games with moves in $\mathbb{R}$ is equivalent to the statement that $M^\sharp_n(\mathbb{R})$ exists for all $n\in\mathbb{N}$, and that determinacy of all projective games of length $\omega^2$ with moves in $\mathbb{N}$ is equivalent to the statement that $M^\sharp_n(\mathbb{R})$ exists and satisfies $\ad$ for all $n\in\mathbb{N}$.
\end{abstract}
\date{\today}
\clearpage
\maketitle

\setcounter{tocdepth}{1}
%\tableofcontents
\section{Introduction}
In this article, we study the interplay between the determinacy of infinite games on $\mathbb{R}$ whose payoff is projective  and canonical transitive models of set theory that contain $\mathbb{R}$ and finitely many Woodin cardinals. 

Given a set $x$, let $M^\sharp_n(x)$ denote the minimal $\omega_1$-iterable active $x$-premouse with $n$ Woodin cardinals, if it exists, and let $M_n(x)$ denote the proper-class model obtained from $M_n^\sharp(x)$ by iterating its topmost measure $\Ord$-many times. 
A theorem of Woodin states that projective games on $\mathbb{N}$ are determined if, and only if, $M^\sharp_n(x)$ exists for every $n$ and every $x\in\mathbb{R}$. Woodin's proof of determinacy from the existence of $M^\sharp_n(x)$ is unpublished, but the result was later refined by Neeman \cite{Ne04}, who proved local versions of the theorem at each level of the projective hierarchy. 
A proof of the other implication can be found in \cite{MSW}. We refer the reader to Larson \cite{La} and the introduction of Neeman \cite{Ne04} for historical background.

The same way that the model $M_n$ generalizes $L$, the model $M_n(\mathbb{R})$ generalizes $L(\mathbb{R})$. An issue related to projective determinacy for games on $\mathbb{R}$ is that of identifying the sets of reals that are $\Sigma_1$-definable over $M_n(\mathbb{R})$.\footnote{By this, we mean $\Sigma_1$-definable in the language of set theory with additional predicates for $\mathbb{R}$ and the extender sequence of $M_n(\mathbb{R})$.} We denote this collection by $\Sigma_1^{M_n(\mathbb{R})}$. For other related models, the following characterizations are well known:

\begin{theorem}\
\begin{enumerate}
\item (Shoenfield) $\Sigma_1^L = \Sigma^1_2$.
\item (Steel \cite{St95}) $\Sigma_1^{M_{2n}} = \Pi^1_{2n+2}$; $\Sigma_1^{M_{2n+1}} = \Sigma^1_{2n+3}$.
\item (Martin-Steel \cite{MaSt08}) $\Sigma_1^{L(\mathbb{R})} = \Game^\mathbb{R}\Pi^1_1$.
\end{enumerate}
\end{theorem}
\noindent Here, $\Game^\mathbb{R}$ denotes the real-game quantifier (see the following section).

The characterization of $\Sigma_1^{M_{n}}$ depends on the parity of $n$. Our main theorem is a characterization of $\Sigma_1^{M_n(\mathbb{R})}$ that combines the ones for $M_n$ and $L(\mathbb{R})$, and exhibits similar periodicity:

\begin{theorem}\label{theoremPointclass}
Suppose that $M^\sharp_n(\mathbb{R})$ exists for all $n$. 
\begin{enumerate}
\item If $n$ is even, then $\Sigma_1^{M_n(\mathbb{R})} = \Game^\mathbb{R}\Pi^1_{n+1}$.
\item If $n$ is odd, then $\Sigma_1^{M_n(\mathbb{R})} = \Game^\mathbb{R}\Sigma^1_{n+1}$.
\end{enumerate}
\end{theorem}
Theorem \ref{theoremPointclass} relativizes to real parameters, as do all other results stated in their lightface forms below.

The proof of Theorem \ref{theoremPointclass} has two parts: first, the inclusion from right to left. To the best of our knowledge, the argument here is new; it is similar for even and odd $n$. In particular, it provides a new proof for the result in the case $n = 0$, although it makes use of an additional determinacy hypothesis. The second part is the inclusion from left to right. Here, however, we need slightly different arguments for even and odd $n$. The one for even $n$ is similar to the case $n=0$, but makes use of the games from \cite{AgMu}. The argument for odd $n$ requires a variation of these games. This is all done in Section \ref{SectSigma1} under an additional determinacy assumption.

In Section \ref{SectApplications}, we present applications of the characterization and its proof.
First, we prove an analogue of the Neeman-Woodin theorem for games on $\mathbb{R}$ which, combined with the results from the preceding section yields Theorem \ref{theoremPointclass}:
\begin{theorem}\label{theoremGamesOnR}
The following are equivalent:
\begin{enumerate}
\item Projective determinacy for games on $\mathbb{R}$;
\item $M_n^\sharp(\mathbb{R})$ exists for all $n\in\mathbb{N}$.
\end{enumerate}
\end{theorem}
We suspect that, as is the case for games on $\mathbb{N}$, the result can be improved to yield equivalences at each level of the projective hierarchy:
\begin{conjecture}\label{ConjectureMnR}
Let $n\in\mathbb{N}$. The following are equivalent:
\begin{enumerate}
\item $\PI^1_{n+1}$-determinacy for games on $\mathbb{R}$;
\item $M_n^\sharp(\mathbb{R})$ exists.
\end{enumerate}
\end{conjecture}
We remark that the instance of Conjecture \ref{ConjectureMnR} when $n=0$ is true, by theorems of Martin \cite{Ma70} and Trang \cite{Tr13}.

We obtain a similar equivalence that results from augmenting the existence of $M_n(\mathbb{R})$ with the assertion that it satisfies the Axiom of Determinacy: \begin{theorem}\label{theoremLongGames}
The following are equivalent:
\begin{enumerate}
\item Projective determinacy for games on $\mathbb{N}$ of length $\omega^2$;
\item $M_n^\sharp(\mathbb{R})$ exists and satisfies $\ad$ for all $n\in\mathbb{N}$.
\end{enumerate}
\end{theorem}
We remark that Conjecture \ref{ConjectureMnR} implies its counterpart for Theorem \ref{theoremLongGames} (this is by our proof of Theorem \ref{theoremLongGames}).

\section{Preliminaries}
\label{SectPreliminaries}
We study infinite, two-player perfect-information games of the form
\[ \begin{array}{c|ccccc} \mathrm{I} & x_0 & & x_2 & &\hdots \\ \hline
    \mathrm{II} & & x_1 & & x_3 & \hdots 
   \end{array} \]
   where each move $x_i$ is an element of  $\omega^\omega$---a \emph{real number}.
In the game above, the payoff set is given by a subset of $(\omega^\omega)^{\omega}$, i.e., a collection of sequences of reals of length $\omega$. Observe that, making use of a fixed (e.g., recursive) bijection between $\omega$ and $\omega\times\omega$, one can code elements of $(\omega^\omega)^\omega$ by elements of $\omega^\omega$, and thus each subset of $(\omega^\omega)^{\omega}$ can be coded by a subset of $\omega^\omega$. If $x\in\omega^\omega$, we denote by $x_{(i)}$ the $i$th real coded in this way. 

If $A\subset \omega^\omega\times\omega^\omega$, we denote by $\Game^\mathbb{R}A$ the set of all $x\in\omega^\omega$ such that Player I has a winning strategy in the game with payoff 
\[\{(y_{(0)}, y_{(1)}, \hdots): (x,y) \in A\}.\]
Similarly for other spaces.
If $\Gamma$ is a pointclass, we denote by $\Game^\mathbb{R}\Gamma$ the pointclass of all $\Game^\mathbb{R}A$ such that $A\in \Gamma$.

\subsection{Topological Remarks} 
We are interested in subsets of $(\omega^\omega)^{\omega}$ whose codes are definable over $\omega^\omega$ with parameters---the projective sets. These are the sets that can be obtained from open subsets of $\omega^\omega$ (in the product topology) by applying finitely many projections and complements.
These are not the same sets which are projective when $(\omega^\omega)^\omega$ is regarded as a product of infinitely many discrete copies of $\omega^\omega$, although  proving the determinacy of these sets does not require greater consistency strength, as we shall see later. To see that it does not require less consistency strength,
we need the following simple observation:

\begin{lemma}\label{LemmaTopologies}
Let $\mathcal{T}$, $\mathcal{S}$ be topologies on a set $X$ and let $\mathcal{T}^\omega$ and $\mathcal{S}^\omega$ be the product topologies.
Suppose that $\mathcal{T}\subseteq\mathcal{S}$; then $\mathcal{T}^\omega\subseteq\mathcal{S}^\omega$.
\end{lemma}
\proof
It suffices to verify that basic open sets of $\mathcal{T}^\omega$ are open in $\mathcal{S}^\omega$. These sets are of the form
\[U = \{(x_0,x_1,\hdots) \in X^\omega: x_i \in U_i \text{ for } i\leq n\},\]
where $n\in\mathbb{N}$ and each of $U_0,\hdots, U_n$ is $\mathcal{T}$-open, thus $\mathcal{S}$-open, so $U$ is also $\mathcal{S}^\omega$-open.
\endproof

\begin{remark}\label{RemarkProjective}
Let $X$ be a discrete space. Thus, a set $A\subset X^\omega$ is closed if, and only if, it is the set of branches through some tree $T$ on $X$. We may define a subset of $X$ to be $\SIGMA^1_1$ if it is the projection of a tree on $X\times\omega$, and go on to define the projective hierarchy as usual. In particular, we might focus on the case $X = \omega^\omega$. By Lemma \ref{LemmaTopologies}, every set which is ``projective'' in this sense is also projective in the usual sense. We state this as a corollary:
\end{remark}

\begin{corollary}\label{CorollaryTopologies}
Let $X$ be the space $(\omega^\omega)^\omega$, viewed as a product of $\omega^\omega$ with the usual topology (so $X$ is homeomorphic to $\omega^\omega$); and let $Y$ be the space $(\omega^\omega)^\omega$ viewed as a product of discrete spaces. Then, every open (Borel, analytic, projective) set in $X$ is also open (Borel, analytic, projective) in $Y$.
\end{corollary}
\proof
Simply take $\mathcal{T}$ to be the usual topology on $\omega^\omega$ and $\mathcal{S}$ to be the discrete topology, and apply Lemma \ref{LemmaTopologies}.
\endproof

\subsection{Inner Model Theory} 
We work with canonical, fine
structural models with large cardinals, i.e., \emph{premice}. We
refer the reader to e.g., Steel \cite{St10} for an introduction, and to
Mitchell-Steel \cite{MS94}, Schindler-Steel-Zeman \cite{SchStZe02}, and Steel \cite{St08b} for additional
background. For basic set-theoretic definitions and results we refer to
Kanamori \cite{Ka08} and Moschovakis \cite{Mo09}. 

We will use Mitchell-Steel indexing for
extender sequences and the notation from Steel \cite{St10}. We will consider relativized premice constructed over
transitive sets of reals $X$ as in Steel \cite{St08b}. We denote by
$\lpm = \{\dot \in, \dot E, \dot F, \dot X\}$ the language of
relativized premice, where $\dot E$ is the predicate for the extender
sequence, $\dot F$ is the predicate for the top extender, and $\dot X$
is the predicate for the set over which we construct the premouse.

We say an \emph{$X$-premouse}
$M = (J_\alpha^{\vec{E}}, \in, \vec{E}, E_\alpha, X)$ for
$\vec{E} = (\dot E)^M$, $E_\alpha = (\dot F)^M$, and $X = (\dot X)^M$
is \emph{active} if $E_\alpha \neq \emptyset$. Otherwise, we say $M$
is \emph{passive}. We let
$M | \gamma = (J_\gamma^{\vec{E}}, \in, \vec{E} \upharpoonright
\gamma, E_\gamma, X)$ for $\gamma \leq M \cap \Ord$ and write $M || \gamma$ for the passive initial segment of $M$ of height
$\gamma$. In particular,
$M||\Ord^M$ denotes the premouse $N$ which agrees with $M$ except that
we let $(\dot F)^N = \emptyset$. 
% We say an ordinal $\eta$ is a\emph{(strong) cutpoint} of $M$ if there is no extender $E$ on the$M$-sequence with $\crit(E) \leq \eta \leq \lth(E)$. 
For an $X$-premouse, we denote by $\Sigma_1^M$ the pointclass of all sets  that are $\Sigma_1$-definable over $M$ in the language $\lpm$.
If $M$ is a premouse, we sometimes abuse notation by referring to other premice extending $M$ as ``$M$-premice.''

We will make use of the notions of \emph{$\alpha$-iterability}, where $\alpha$ is an ordinal, and its projective variants, $\Pi^1_n$-iterability. A definition of the latter can be found in Steel \cite{St95}. To make the notation uniform, we say a premouse is \emph{$\Pi^1_1$-iterable} if it is wellfounded. We will also make use of the fine structure of premice and the usual comparison theorem, which can be found in Steel \cite{St10}.
Let $A$ be a real number or a set of real numbers. 
We denote by $M_n^\sharp(A)$ the unique smallest sound, $\omega_1$-iterable $A$-premouse which is not $n$-small, if it exists. $M_n(A)$ denotes the proper-class model that results from iterating the top extender of $M_n(A)$ a proper-class amount of times. Let us recall the degrees of correctness of the models $M_n^\sharp(x)$:
Assuming $M_{2n}^\sharp(x)$ exists, it is $\SIGMA^1_{2n+2}$-correct. Similarly, $M^\sharp_{2n+1}(x)$ is $\SIGMA^1_{2n+2}$-correct, but can compute $\SIGMA^1_{2n+3}$-truth by forcing over its collapse algebra. (This is a result of Neeman. The result from $(\omega_1+1)$-iterability is due to Woodin; a proof can be found in \cite{MSW}.) Thus, every $\Sigma^1_{2n+3}(x)$ set of reals has a member recursive in $M^\sharp_{2n+1}(x)$, so every model closed under $x\mapsto M^\sharp_{2n+1}(x)$ is $\SIGMA^1_{2n+3}$-correct (this is due to Woodin; see Steel \cite{St95}). The arguments also apply to the models $M_n^\sharp(A)$, where $A$ is a countable set of real numbers, and show that it is $\SIGMA^1_{n+2}$-correct if $n$ is even, or $\SIGMA^1_{n+1}$-correct if $n$ is odd. Finally, they apply to generic extensions of these models by small partial orders, say, of size smaller than the least measurable cardinal. In particular, if $g$ is an $M_{2n}^\sharp(A)$-generic wellordering of $A$, then $M_{2n}^\sharp(A)[g]$ is $\SIGMA^1_{n+2}$-correct.

In our definition of $M_n^\sharp(A)$, we only demand $\omega_1$-iterability, rather than $(\omega_1+1)$-iterability. This is because often $\omega_1$-iterability is all one can hope to obtain as a consequence of determinacy in $V$. We refer the reader to \cite{MSW} for details on how to survive on a budget of strategies for countable trees. 

Since we will be interested in the case $X = \mathbb{R}$, we do not require $X$ to be countable. Here, we need a notion of iterability for $X$-premice that suffices to carry out all relevant comparison arguments: we say that an $X$-premouse is \emph{countably iterable} if all its countable elementary substructures are $\omega_1$-iterable.\footnote{This is non-standard, in the sense that one usually requires $(\omega_1+1)$-iterability here.} Similarly, we say that an $X$-premouse is \emph{countably $\Pi^1_n$-iterable} if all its countable elementary substructures are $\Pi^1_n$-iterable.

\section{The $\Sigma_1$ Subsets of $M_n(\mathbb{R})$}\label{SectSigma1}
In this section, we shall characterize the pointclass of all sets which are $\Sigma_1$-definable over $M_n(\mathbb{R})$, in the sense of the previous section.
The characterization depends on whether $n$ is even or odd. Our proof of the local version of the theorem requires determinacy for games on reals. 

\begin{theorem}\label{TheoremSigma1}
Let $n\in\mathbb{N}$. Suppose that $M_n^\sharp(\mathbb{R})$ exists and that $\Pi^1_{n+1}$ games on $\mathbb{R}$ are determined.
\begin{enumerate}
\item If $n$ is even, then $\Sigma_1^{M_n(\mathbb{R})} = \Game^\mathbb{R}\Pi^1_{n+1}$.
\item If $n$ is odd, then $\Sigma_1^{M_n(\mathbb{R})} = \Game^\mathbb{R}\Sigma^1_{n+1}$.
\end{enumerate}
\end{theorem}

We shall prove the theorem by considering each inclusion separately in the cases for even $n$ and odd $n$.

\begin{definition}
Suppose $G$ is a game on reals and $A
\subset\mathbb{R}$.
We denote by $G_{A}$ the  modification of $G$ in which we demand of each player that $x \in A$ for every move $x$.
\end{definition}
Thus, $G = G_{\mathbb{R}}$, and for $A
\subset\mathbb{R}$, the winning set for Player I in $G_A$ is $\{x \in \mathbb{R}: $ either $x_{(n)}\not \in A$ for some $n\in\mathbb{N}$ and the least such $n$ is odd, or $x_{(n)} \in A$ for all $n\in\mathbb{N}$ and $x$ is a winning play for Player I in $G\}$. 

Towards proving Theorem \ref{TheoremSigma1}, the following lemma allows us to define $\Game^\mathbb{R}\Pi^1_{n+1}$ sets in a $\Sigma_1$ way over $M_n(\mathbb{R})$ when $n$ is even.
\begin{lemma}\label{LemmaSigma1Eequiv}
Let $n\in\mathbb{N}$ be even and suppose that $M_n^\sharp$ exists and that $\Pi^1_{n+1}$ games on $\mathbb{R}$ are determined. 
Then, for every $\Pi^1_{n+1}$ game $G$ with moves in $\mathbb{R}$,  the following are equivalent:
\begin{enumerate}
\item Player I has a winning strategy for $G$, and
\item \label{cond2ClaimSigma1} $M_n(\mathbb{R})\models \text{``$\Vdash_{\coll(\omega,\mathbb{R})}$ Player I has a winning strategy for $G_{\mathbb{R}^V}$.''}$
\end{enumerate}
\end{lemma}
\begin{proof}
Fix a game $G$ with $\Pi^1_{n+1}$ payoff.
We will work with $G$ and with games of the form $G_A$.
If $A$ is countable, and $x_A$ codes $A$, say
\[A = \{x_{(i)}: i\in\mathbb{N}\},\]
we will also consider the further variant of $G$ in which instead of playing reals in $A$, players play their indices in $x_A$. We denote this variant by $G(x_A)$. 

We observe that for a set $A \in \mathcal{P}_{\omega_1}(\mathbb{R})$, Player I has a winning strategy in $G_A$ if, and only if, she has one in $G({x_A})$, for any $x_A$ as above. Suppose that this is the case for some $A \in \mathcal{P}_{\omega_1}(\mathbb{R})$ and some $x_A \in\mathbb{R}$.
By definition, $G(x_A)$ is a $\Pi^1_{n+1}(x_A)$ game on $\mathbb{N}$. Player I having a winning strategy for $G(x_A)$ is expressible by a formula in
\[\Game\Pi^1_{n+1}(x_A) = \Sigma^1_{n+2}(x_A).\]
\begin{claim}
The following are equivalent:
\begin{enumerate}
\item Player I has a winning strategy for $G(x_A)$, and
\item $M_n(A)\models \text{``$\Vdash_{\coll(\omega,A)}$ Player I has a winning strategy for $G_A$.''}$
\end{enumerate}
\end{claim}
\begin{proof}
This follows by absoluteness.
\end{proof}

Now, suppose Player I has a winning strategy $\sigma$ for $G$. Let $W$ be the transitive collapse of a countable elementary substructure of some large $V_\kappa$ and $A = \mathbb{R}^W$. %; thus, $G_A$ is $W$'s version of $G$.
By elementarity, there is a winning strategy $\sigma'$ for $G_A$ in $W$, and $\sigma'$ agrees with $\sigma$ on moves in $W$. $\sigma'$ can be extended to a total strategy $\bar\sigma$ for $G_A$ in $V$ by assigning arbitrary responses if  Player II ever makes a move outside of $A$.
Because it agrees with $\sigma$ on moves in $W$, $\bar\sigma$ instructs Player I to play reals in $A$ as long as Player II does, so every full play by $\bar\sigma$ in which Player II has played only reals in $A$ will be a winning move in $G_A$. It follows that $\bar\sigma$ is a winning strategy for $G_A$ in $V$ for Player I. By the claim, 
\[M_n(A)\models \text{``$\Vdash_{\coll(\omega,A)}$ Player I has a winning strategy for $G_A$.''}\]
By the elementarity between $V_\kappa$ and $W$, $M_n^\sharp(\mathbb{R})^W$ is the transitive collapse of a countable elementary substructure of $M_n^\sharp(\mathbb{R})$ and is therefore $\omega_1$-iterable. Thus,
\[M_n^\sharp(\mathbb{R})^W = M_n^\sharp(\mathbb{R}^W) = M_n^\sharp(A),\]
so, by elementarity,
\[M_n(\mathbb{R})\models \text{``$\Vdash_{\coll(\omega,\mathbb{R})}$ Player I has a winning strategy for $G_{\mathbb{R}^V}$.''}\]
%The existence of such an $N$ is $\Sigma_1^{M_n(\mathbb{R})}$. To complete the proof of the lemma, it suffices to assume that there is some $N\lhd M_n(\mathbb{R})$ as above
%	with $n$ Woodin cardinals such that:
%\[N\models \text{``$\Vdash_{\coll(\omega,\mathbb{R})}$ Player I has a winning strategy for $G$''}\]
Conversely, suppose that
\[M_n(\mathbb{R})\models \text{``$\Vdash_{\coll(\omega,\mathbb{R})}$ Player I has a winning strategy for $G_{\mathbb{R}^V}$.''}\]
We argue that Player I has a winning strategy in $G$. 

Let $W$ be the transitive collapse of a countable elementary substructure of some large $V_\kappa$ and $A = \mathbb{R}^W$. Then, 
\[M_n(A)\models \text{``$\Vdash_{\coll(\omega,A)}$ Player I has a winning strategy for $G_A$.''}\]
Leting $g\subseteq\coll(\omega,A)$ be $M$-generic, with $g\in V$,
\[M_n(A)[g]\models \text{``Player I has a winning strategy $\tau$ for $G(g)$.''}\]
Since $M_n(A)[g]$ has $n$ Woodin cardinals and is thus $\PI^1_{n+1}$-correct, and the statement that $\tau$ is winning for Player I is $\Pi^1_{n+1}(\tau,g)$, $\tau$ really is a winning strategy for $G(g)$. Thus, Player I has a winning strategy for $G_A$ in $V$.

Suppose towards a contradiction that Player I does not have a winning strategy for $G$. By our assumption on the determinacy of $\Pi^1_{n+1}$ games on $\mathbb{R}$, Player II has one, say $\sigma$. By elementarity, Player II has one for $G_A$ in $W$, say $\bar\sigma$, and $\bar\sigma$ agrees with $\sigma$ on moves in $W$, so it follows that $\bar\sigma$ is a winning strategy for Player II for $G_A$ in $V$ as well, which is the desired contradiction.
\end{proof}

\begin{lemma}
Let $n \in\mathbb{N}$ be even.
Suppose that $M^\sharp_n(\mathbb{R})$ exists and that $\Pi^1_{n+1}$ games on $\mathbb{R}$ are determined. Then 
\[\Game^\mathbb{R}\Pi^1_{n+1}\subseteq \Sigma_1^{M_n(\mathbb{R})}.\]
\end{lemma}
\begin{proof}
Let $B \in \Game^\mathbb{R}\Pi^1_{n+1}$, say
\[B = \{x\in\mathbb{R}: \text{ Player I has a winning strategy in the game on reals with payoff $P_x$}\},\]
for some $\Pi^1_{n+1}$ set $P \subset \mathbb{R}\times\mathbb{R}$. Let us uniformly denote by $G_x$ the game on reals with payoff $P_x$.
By Lemma \ref{LemmaSigma1Eequiv},
\begin{align*}
B &= \Big\{x\in\mathbb{R}: M_n(\mathbb{R})\models \text{``$\Vdash_{\coll(\omega,\mathbb{R})}$ Player I has a winning strategy for $(G_{x})_{\mathbb{R}^V}$.''}\Big\}\\
&= \Big\{x\in\mathbb{R}: \exists M\lhd M_n(\mathbb{R})\,\\
&\qquad \qquad M\models \text{``$\Vdash_{\coll(\omega,\mathbb{R})}$ Player I has a winning strategy for $(G_{x})_{\mathbb{R}^V}$.''}\Big\}.
\end{align*}
The existence of such an $M$ is $\Sigma_1^{M_n(\mathbb{R})}$, so $B$ is $\Sigma_1^{M_n(\mathbb{R})}$.
\end{proof}

\begin{lemma}
Let $n\in\mathbb{N}$ be odd.
Suppose that $M^\sharp_n(\mathbb{R})$ exists and that $\Pi^1_{n+1}$ games on $\mathbb{R}$ are determined. Then,
\[\Game^\mathbb{R}\Sigma^1_{n+1}\subseteq \Sigma_1^{M_n(\mathbb{R})}.\]
\end{lemma}
\begin{proof}
This is similar to the even case. One has to
verify the analogue of Lemma \ref{LemmaSigma1Eequiv} -- that the following are equivalent for a $\Sigma^1_{n+1}$ game on reals $G$:
\begin{enumerate}
\item Player I has a winning strategy for $G$, and
\item $M_n(\mathbb{R})\models \text{``$\Vdash_{\coll(\omega,\mathbb{R})}$ Player I has a winning strategy for $G_{\mathbb{R}^V}$.''}$.
\end{enumerate}
Suppose thus that $G$ is $\Sigma^1_{n+1}$ (i.e., that the winning set for Player I is $\Sigma^1_{n+1}$), so that the games $G(x_A)$ are $\Sigma^1_{n+1}(x_A)$. The hypothesis implies $\PI^1_{n+1}$-determinacy for games on $\mathbb{N}$, so, using that $n+1$ is even, we may appeal to the Third Periodicity Theorem (see Moschovakis \cite[Corollary 6E.2]{Mo09}) to deduce that if I wins $G(x_A)$, then I has a winning strategy in $\Delta^1_{n+2}(x_A)$. All such reals belong to $M_n(x_A)$, by a theorem of Woodin (see Steel \cite[Theorem 4.8]{St95}).
The rest of the proof is as before.
\end{proof}

{\begin{center}
{{---}{---}}
\end{center}
}

To complete the proof of the theorem, it remains to prove the converse inclusions in the cases that $n$ is even and odd. 
We will need the following lemma:
\begin{lemma}\label{LemmaClubReals}
Suppose either that $\PI^1_{n+1}$ games on $\mathbb{R}$ are determined or that $M_n^\sharp(\mathbb{R})$ exists. Then, there is a closed, cofinal subset $C$ of $\mathcal{P}_{\omega_1}(\mathbb{R})$ such that if $A \in C$, then 
\[M_n^\sharp(A)\cap\mathbb{R} = A.\]
\end{lemma}
\begin{proof}
The conclusion is proved under the assumption of $\PI^1_{n+1}$-determinacy for games on $\mathbb{R}$ in \cite[Lemma 3.8]{AgMu}; there, it is stated under the stronger assumption of determinacy for $\PI^1_{n+1}$-determinacy for games of length $\omega^2$, but the proof only requires $\PI^1_{n+1}$-determinacy for games on $\mathbb{R}$. In the other case, there is a closed cofinal subset of $\mathcal{P}_{\omega_1}(\mathbb{R})$ consisting of sets $A$ which are sets of reals of transitive collapses of countable elementary substructures $N_A$ of some large $V_\kappa$, so that $M_n^\sharp(\mathbb{R})^{N_A}$ embeds elementarily into $M_n^\sharp(\mathbb{R})$ and thus is $\omega_1$-iterable and equal to $M_n^\sharp(A)$. Since clearly 
\[M_n^\sharp(\mathbb{R})\cap\mathbb{R} = \mathbb{R},\]
the result follows.
\end{proof}

Let us begin with the case of even $n$:
\begin{lemma}\label{LemmaConverseInclusionEven}
Let $n\in\mathbb{N}$ be even. Suppose that $M_n^\sharp(\mathbb{R})$ exists. Then,
\[\mathcal{P}(\mathbb{R}) \cap \Sigma_1^{M_n(\mathbb{R})} \subseteq \Game^\mathbb{R}\Pi^1_{n+1}.\]
\end{lemma}

Lemma \ref{LemmaConverseInclusionEven} is an immediate consequence of Lemma \ref{LemmaGphiLP} below.
Its proof
makes use of model games for premice very similar to the ones in \cite{AgMu}. Here, Players I and II play a sequence of reals and a theory in the language $\mathcal{L}_{pm}(\{\dot x_i:i\in\omega\})$ of premice with added constant symbols $\dot x_i$, for $i\in\omega$. The rules of the game ensure that any model $M$ of the theory contains all reals played during the game and that there is a minimal model, which consists of all elements of $M$ definable from real parameters (the \emph{definable closure} of the reals of $M$).

Fix recursive bijections $m(\cdot)$ and $n(\cdot)$ assigning an odd number
$>1$ to each $\langpm$-formula $\varphi$ such that $m$ and $n$ have
disjoint recursive ranges and for every $\varphi$, $m(\varphi)$ and
$n(\varphi)$ are larger than
$\max\{ i \, \colon \, \dot x_i \text{ occurs in } \varphi\}$. Fix some enumeration
  $(\phi_i \colon i \in \omega)$ of all $\langpm$-formulae such that
  $\dot x_i$ does not appear in $\phi_j$ if $j \leq i$.  Finally, fix a ternary formula $\theta(\cdot,\cdot,\cdot)$ which for any $\mathbb{R}^M$-premouse $M$ uniformly defines the graph of an $\mathbb{R}^M$-indexed family of class-sized wellorders whose union contains every set in $M$ (see Steel \cite[Proposition 2.4.1]{St08b} for an example of such a formula).

\begin{definition}\label{def:Gvarphipsi}
  Let $\varphi$ be a $\lpm$-formula and $n\in\mathbb{N}$ be even; we describe a game
  $G^{\text{even}}_{\varphi}$ of length
  $\omega$ on $\mathbb{R}$. 
  A typical
  run of $G^{\text{even}}_{\varphi}$ looks as follows:
  \[ \begin{array}{c|ccccc} \mathrm{I}  & v_0, x_0 & & v_1,
      x_2 & & \hdots \\ \hline \mathrm{II}  & & x_1 & & x_3 &
      \hdots \end{array} \]
Here, Players I and II take turns, respectively playing reals $(v_i, x_{2i})$ and $x_{2i+1}$, for
    $i \in \omega$. We ask that $v_i \in \{0,1\}$.

  Here $v_i$ will be interpreted as the truth value of the formula
  $\phi_i$ from the enumeration fixed above. This can be thought of as
  Player I either accepting or rejecting the formula $\phi_i$. If so,
  the play determines a complete theory $T$ in the language $\langpm$.

  Player I wins the game $G^{\text{even}}_{\varphi}$ if, and only if,
  \begin{enumerate}
  \item \label{rule:reals} For each $i \in \omega$, $T$ contains the
    sentence $\dot x_i \in \mathbb{R}$ and, moreover, for each
    $j,m \in \omega$, $T$ contains the sentence $\dot x_i(m) = j$ if,
    and only if, $x_i(m) = j$.
  \item \label{rule:dc} For every formula $\phi(x)$ with one free
    variable in the language $\langpm$, and $m(\phi)$ and
    $n(\phi)$ as fixed above, $T$ contains the statements
    \[ \exists x\, \phi(x) \to \exists x\, \exists \alpha \,
      (\phi(x) \land \theta(\alpha, \dot x_{m(\phi)}, x)), \]
    \[ \exists x\, (\phi(x) \land x \in \dot X) \to \phi(\dot
      x_{n(\phi)}). \]
    % [This will ensure that the definable closure of
    % $\{x_i : i \in \omega\}$ in every model $\cM$ of $T$ is an
    % elementary submodel of $\cM$.]
  \item \label{rule:pm} $T$ is a complete, consistent theory such that
    for every countable model $\cM$ of $T$ and every model $\cN^*$
    which is the definable closure of $\{x_i : i < \omega\}$ in
    $\cM \upharpoonright \lpm$, $\cN^*$ is well-founded, and if $\cN$
    denotes the transitive collapse of $\cN^*$,
    \begin{enumerate}
    \item $\cN$ is a sound and solid $n$-small $X$-premouse, where
      $X = \{x_i \colon i \in \omega\}$,
    \item $\cN$ is $\Pi^1_{n+1}$-iterable (in the sense of \cite[Definition
      1.6]{St95}), and
    \item \label{rule:pmPsi} $\cN \vDash \varphi$, but no proper initial segment of $\cN$ satisfies $\varphi$.
  \end{enumerate}
\end{enumerate}
If Player I plays according to all these rules, he wins the game. In
this case there is a unique premouse $\cN_p$ as in (\ref{rule:pm})
associated to the play $p = (v_0,x_0, x_1, v_1, x_2, \dots)$ of
the game. Otherwise, Player II wins.
\end{definition}

Observe that the winning set for Player I in the game $G^{\text{even}}_{\phi}$ is $\Pi^1_{n+1}$.
Motivated by the definition of $G^{\text{even}}_{\varphi}$ and the argument of \cite[Theorem 3.1]{AgMu}, we introduce the following definition for the proof:
\begin{definition}
Let $a\in\mathcal{P}_{\omega_1}(\mathbb{R})$ and $\varphi$ be a formula in the language $\lpm$.
We say that an $a$-premouse is a \emph{$\varphi$-witness} if it satisfies $\varphi$ and a \emph{minimal $\varphi$-witness} if, in addition, it has no proper initial segment satisfying $\varphi$.
\end{definition}

\begin{lemma}\label{LemmaProjectsToR}
If an $a$-premouse $M$ satisfies a formula $\varphi$ but no proper initial segment of $M$ satisfies $\varphi$, then $J_1(M)$ projects to $a$.
\end{lemma}
\begin{proof}
This is because $J_1(M)$ satisfies a $\Sigma_1$ fact that holds in none of its proper initial segments (viz. the fact that some proper initial segment satisfies $\varphi$).
\end{proof}

\begin{lemma}\label{LemmaGphiLP}
Let $n$ be even and assume either $\PI^1_{n+1}$-determinacy for games on reals or that $M_n^\sharp(\mathbb{R})$ exists. Then,
Player I has a winning strategy in $G^{\text{even}}_{\varphi}$ if, and only if, there is a sound, countably iterable, $n$-small $\mathbb{R}$-premouse satisfying $\varphi$.
\end{lemma}
\begin{proof}
Both of the possible hypotheses imply $\PI^1_{n+1}$-determinacy for games on $\mathbb{N}$, which implies that the conclusion of the Comparison Theorem holds for every $a \in \mathcal{P}_{\omega_1}(\mathbb{R})$ and every pair of countable $n$-small, $\omega_1$-iterable $a$-premice with no Woodin cardinals or which project to $a$ (see e.g., \cite{MSW} for a proof of this).

If there is a premouse as in the statement of the lemma, then there is a least one, say $M$ (this follows from countable iterability, the remark just given, and Lemma \ref{LemmaProjectsToR}). Player I can win $G^{\text{even}}_{\varphi}$ by playing the theory of $M$.

The converse implication is proved by an argument similar to the one of \cite[Theorem 3.1]{AgMu}, but there are a few differences, so we include the proof for the reader's convenience. The idea for the proof is to assume that Player I has a winning strategy $\sigma$ for $G^{\text{even}}_{\varphi}$ and use the strategy to construct many $A$-premice, for various $A$, which will then be joined together into an $\mathbb{R}$-premice. Because $\sigma$ is a winning strategy, the $A$-premice will have the properties that we desire of the $\mathbb{R}$-premice, and these will transfer by elementarity. This idea of ``joining together'' $A$-premice to produce an $\mathbb{R}$-premouse comes from Martin-Steel \cite[Lemma 3]{MaSt08}.

On to the proof. Suppose Player I has a winning strategy $\sigma$.  We claim that there is a closed, cofinal $C\subseteq\mathcal{P}_{\omega_1}(\mathbb{R})$ such that for all $a\in C$ some sound $\Pi^1_{n+1}$-iterable, $n$-small $a$-premouse satisfies $\varphi$.
To see this, we argue as in the proof of \cite[Lemma 3.8, Case 1]{AgMu}: let $W$ be the transitive collapse of a countable elementary substructure $Y$ of some large $V_\kappa$ such that $\sigma \in Y$. It suffices to show that such a premouse exists for $a =\mathbb{R}^W$. Letting $\bar\sigma$ be the preimage of $\sigma$ under the collapse embedding, we have
\[W\models \text{``$\bar\sigma$ is a winning strategy for I in $G_\varphi^{\text{even}}$.''}\]
Let $h$ be a wellordering of $a$ in order-type $\omega$ and consider the play $p$ of $G_\varphi^{\text{even}}$ in which Player II enumerates $h$ and Player I responds according to $\sigma$. All proper initial segments of this play belong to $W$, since $\bar\sigma \in W$ and $\bar\sigma$ agrees with $\sigma$ on its domain. It follows that the reals played in $p$ are exactly those in $a$, and so, since $\sigma$ is a winning strategy, $p$ determines a $\Pi^1_{n+1}$-iterable $a$-premouse as desired. As claimed, there is a closed, cofinal $C\subseteq\mathcal{P}_{\omega_1}(\mathbb{R})$ such that for all $a\in C$ some sound $\Pi^1_{n+1}$-iterable, $n$-small $a$-premouse satisfies $\varphi$. Moreover, each of these $a$-premice has an initial segment none of whose proper initial segments satisfies $\varphi$.

Suppose $M$ is such an $a$-premouse. Then, $J_1(M)$ projects to $a$ by Lemma \ref{LemmaProjectsToR}. 
By Steel \cite[Lemma 3.3]{St95}, it is thus an initial segment of $M_n^\sharp(a)$ (this is where we use that $n$ is even). Thus, for each $a \in C$, there is some $M\lhd M_n^\sharp(a)$ which satisfies $\varphi$. Appealing to Lemma \ref{LemmaClubReals} (this is where the assumption of either $\PI^1_{n+1}$-determinacy for games on reals or the existence of $M_n^\sharp(\mathbb{R})$ is used), and refining $C$ if necessary, we may assume that
\[\mathbb{R}\cap M_n^\sharp(a) = a\]
for all $a \in C$.

For the remainder of this proof, if $p$ is a sequence of reals, we denote by $p^*$ the set of all reals appearing in $p$.
Now, if $p$ is a countable sequence of reals, $\sigma$ applied to $p$ yields a (possibly larger) countable sequence of reals and a theory $T_p$ whose minimal model is a $\Pi^1_{n+1}$-iterable, sound, minimal $\varphi$-witness $N_p$. From the conclusion of the previous paragraph follows that by replacing $\varphi$ with the formula 
\[\varphi^* = \text{``}\dot X = \mathbb{R} \wedge \varphi\text{''}\]
and restricting the choice of $p$ to elements of
\[\{p : p^* \in C\},\]
we may assume that $N_p$ is a $\mathbb{R}^{N_p}$-premouse, $\mathbb{R}^{N_p} = p^*$, and  some $M\lhd M_n^\sharp(p^*)$ is a $\varphi$-witness.
Suppose $p$ and $q$ are two sequences of reals obtained via plays of $G^{\text{even}}_{\varphi}$ consistent with $\sigma$ and that $p^*, q^* \in C$. By \cite[Lemma 3.4]{AgMu}, $N_p$ and $N_q$ are $\omega_1$-iterable. We observe some facts about them:
\begin{claim}\label{ClaimLemmaGphiLP}
The sequences $p$ and $q$ have the following properties:
\begin{enumerate}
\item \label{ClaimLemmaGphiLP1} $J_1(N_p)$ projects to $p^*$. 
%\item \label{ClaimLemmaGphiLP2} $N_p$ is a initial segment of $M_n^\sharp(p^*)$. 
\item \label{ClaimLemmaGphiLP3} If $p^* = q^*$, then $N_p = N_q$.
\item \label{ClaimLemmaGphiLP4} If $p^* \subseteq q^*$, then there is an elementary embedding from $N_p$ into $N_q$; moreover, the embedding is unique.
\end{enumerate}
\end{claim}
\begin{proof}
\eqref{ClaimLemmaGphiLP1} follows from Lemma \ref{LemmaProjectsToR}, and the fact that $N_p$ is a minimal $\varphi$-witness.
%Otherwise, we have 
%\[\mathbb{R}^{L(N_p)} = \mathbb{R}^{N_p}.\]
%However, $L(N_p)\models$ ``there is an initial segment of me which is a $\phi$-witness.'' Since the existence of a $\phi$-witness is $\Sigma_1$ and
%\[L_{\omega_1^{L(N_p)}}(N_p)\prec_1 L(N_p),\]
%the least $\phi$-witness is countable in $L(N_p)$ and thus belongs to $N_p$, which is impossible.
%Item \eqref{ClaimLemmaGphiLP2} follows from \eqref{ClaimLemmaGphiLP1}: since $N_p$ is sound, $J_1(N_p)$ is also a $p^*$-premouse. Let $P$ be the least initial segment of $M_n^\sharp(p^*)$ which projects to $P^*$ and has an initial segment which is a $\varphi$-witness. Since they are both countable $n$-small, $\omega_1$-iterable $p^*$-premice which project to $p^*$, we may apply the conclusion of the comparison theorem to $J_1(N_p)$ and $P$ to see that
%\[J_1(N_p) \trianglelefteq P,\]
%and thus $N_p \lhd M_n^\sharp(p^*)$.
\eqref{ClaimLemmaGphiLP3} is immediate from the minimality property of $N_p$ and $N_q$ and noting that, by \eqref{ClaimLemmaGphiLP1}, $J_1(N_p)$ and $J_1(N_q)$ must be equal.
The proof of \eqref{ClaimLemmaGphiLP4} is as in Martin-Steel \cite{MaSt08}; we sketch it:

Suppose otherwise and let $(p,v)$ and $(q,w)$ be plays with real parts $p$ and $q$, respectively (i.e., $v$ and $w$ denote the sequences of truth values played by Player I), such that $(p,v)$ and $(q,w)$ are consistent with $\sigma$, but there is no elementary embedding between $N_p$ and $N_q$. Choose $m$ large enough such that for some formula $\chi(\dot x_0, \hdots, \dot x_i)$,
\begin{enumerate}
\item there are numbers $k_0, k_1,\hdots, k_i\leq m$ such that for each $j\leq i$, the $j$th real of $p$ is the $k_j$th real of $q$, and
\item the formulas $\chi(\dot x_0, \hdots, \dot x_i)$ and $\chi(\dot x_{k_0}, \hdots, \dot x_{k_i})$ are assigned different truth values according to $v$ and $w$.
\end{enumerate}
Since in all turns after the $m$th, Player II is free to play whatever she desires, one can find extensions $(p', v')$ and $(q',w')$ of $(p\upharpoonright m, v\upharpoonright m)$ and $(q\upharpoonright m, w\upharpoonright m)$ in which the same reals have been played. Since the formulas $\chi(\dot x_0, \hdots, \dot x_i)$ and $\chi(\dot x_{k_0}, \hdots, \dot x_{k_i})$ were assigned different truth values, the models corresponding to $(p',v')$ and $(q',w')$ must be different, which contradicts
 \eqref{ClaimLemmaGphiLP3}.
\end{proof}

Now, let $\mathcal D$ be the directed system of all models of the form $N_p$, where $p$ is as above, together with the embeddings given by the claim. The directed system is countably closed, in the sense that for  every countable subset $\mathcal{D}_0$ of $\mathcal{D}$, there is a fixed $N_0 \in \mathcal{D}$ into which every member of $\mathcal{D}_0$ embeds elementarily; it follows that $\mathcal{D}$ has a wellfounded direct limit which we identify with its transitive collapse $M$. Write
\[j_p: N_p \to M\]
for the direct limit elementary embedding.
%, so that, letting $N = j_p[N_p]$ ($=$ the pointwise image of $N_p$ under $j_p$),
%\[j_p^{-1}\upharpoonright N: N \to N_p\]
%is also elementary for each $p$ as above.
By elementarity, $M$ is an $n$-small $\mathbb{R}^M$-premouse satisfying $\varphi$; because the directed system has members containing arbitrary reals, $\mathbb{R}^M = \mathbb{R}$. 
It remains to show that $M$ is countably iterable. We use the fact that, by elementarity, every element of $M$ is  definable over $M$ from a real parameter in $M$.\footnote{By this, we mean that for every $a\in M$, there is a first-order formula $\chi$ in the language $\lpm$ and a real number $x$ in $M$ such that $a$ is the unique element of $M$ satisfying $\chi(a,x)$ in $M$.}

Let $\bar M$ be a countable elementary substructure of $M$. Since every element of $M$ is definable over $M$ from a real parameter, we may assume without loss of generality that
\[\bar M = \{a \in M : a \text{ is definable over $M$ from a parameter in }\bar M \cap \mathbb{R}\}.\]
Let $a\in C$ be such that $\mathbb{R}^{\bar M}\subseteq a$ and $a$ is the set of reals played in a run $p$ of $G^{\text{even}}_{\varphi}$ consistent with $\sigma$ (this $a$ exists e.g., by the argument for the existence of $C$ we gave earlier). 
Since 
\[j_p: N_p \to M\] 
is elementary, $\mathbb{R}^{\bar M}\subseteq a = j_p[a]$ ($=$ the pointwise image of $a$ under $j_p$), and every element of $\bar M$ is definable over $\bar M$ from a real parameter in $\bar M$,
we have $\bar M\subseteq j_p[N_p]$,  and the embedding
\[j_p^{-1}\upharpoonright \bar M: \bar M \to N_p\]
is elementary. Therefore, the transitive collapse of $\bar M$ is $\omega_1$-iterable by the pullback strategy. This proves the lemma.
\end{proof}

Moving on to odd $n$, we show:

\begin{lemma}\label{LemmaConverseInclusionODD}
Let $n\in\mathbb{N}$ be odd. Suppose that $M_n^\sharp(\mathbb{R})$ exists. Then,
\[\mathcal{P}(\mathbb{R}) \cap \Sigma_1^{M_n(\mathbb{R})} \subseteq \Game^\mathbb{R}\Sigma^1_{n+1}.\]
\end{lemma}
The argument here is similar to the preceding one, but there are a few differences. 
We will employ a game $G^{\text{odd}}_{\psi}$ similar to $G^{\text{even}}_{\varphi}$, but engineered so that an appropriate $\mathbb{R}$-premouse can be constructed from a winning strategy for Player II, and not for Player I. The key difference is that the game will require Player I to play the theory of a premouse which is \emph{not} $n$-small. While in the even case we constructed the $\mathbb{R}$-premouse ``from below'' using the small premice given by the strategy, in this case we will instead find the premouse ``from above,'' as an initial segment of $M_n^\sharp(\mathbb{R})$. 

\begin{definition}
Let $\psi$ be a formula in the language $\lpm$ and $n$ be odd. The game $G^{\text{odd}}_{\psi}$  is defined just like $G^{\text{even}}_{\varphi}$, except that 
condition \eqref{rule:pm} is replaced by:
\begin{enumerate}
\item[(3')] \label{Oddrule:pm} $T$ is a complete, consistent theory such that
    for every countable model $\cM$ of $T$ and every model $\cN^*$
    which is the definable closure of $\{x_i : i < \omega\}$ in
    $\cM \upharpoonright \lpm$, $\cN^*$ is well-founded, and if $\cN$
    denotes the transitive collapse of $\cN^*$,
    \begin{enumerate}
    \item $\cN$ is a sound and solid $X$-premouse which is not $n$-small, but all of whose proper initial segments are $n$-small, where
      $X = \{x_i \colon i \in \omega\}$,
    \item $\cN$ is $\Pi^1_{n+1}$-iterable (in the sense of \cite[Definition
      1.4]{St95}),\footnote{It should be pointed out that this iterability assumption in general does not suffice for the conclusion of the comparison theorem to hold between $\cN$ and an $n$-small iterable premouse if $\cN$ is not $n$-small.} and
    \item \label{Oddrule:pmPsi} $\cN$ does not have a proper initial segment satisfying $\psi$.
  \end{enumerate}
\end{enumerate}
\end{definition}

The canonical example of a model satisfying the first two conditions is $M_n^\sharp$, and this requires that $n$ be odd (e.g., $M_2^\sharp$ is not $\Pi^1_3$-iterable; as remarked by Steel \cite[p. 95]{St95}).
Observe that if an $\mathbb{R}^N$-premouse $N$ satisfies the first condition, then it projects to $\mathbb{R}^N$. The winning set for Player I in the game $G^{\text{odd}}_{\psi}$ is $\Pi^1_{n+1}$; hence, the winning set for Player II is $\Sigma^1_{n+1}$. To prove Lemma  \ref{LemmaConverseInclusionODD}, it suffices to show:

\begin{lemma}\label{LemmaHphi}
Let $n$ be odd and assume that $M_n^\sharp(\mathbb{R})$ exists. Then, Player II has a winning strategy in $G^{\text{odd}}_{\psi}$ if, and only if, there is a sound, countably iterable $n$-small $\mathbb{R}$-premouse satisfying $\psi$.
\end{lemma}
\begin{proof}
Since $M_n^\sharp(\mathbb{R})$ exists, $M_n^\sharp(a)$ exists for all $a\in\mathbb{R}$ and all $a\in\mathcal{P}_{\omega_1}(\mathbb{R})$. By Neeman \cite{Ne04}, we have access to $\PI^1_{n+1}$-determinacy for games on $\mathbb{N}$.

If there is a premouse as in the statement of the lemma, then, as before, there is a  least one, say $M$. As noted above, $M$ projects to $\mathbb{R}$. 
Using that $n$ is odd, we see that if $N$ is a countably $\Pi^1_{n+1}$-iterable $\mathbb{R}$-premouse which is sound and tame but not $n$-small and which projects to $\mathbb{R}$, then $M$ is an initial segment of $N$.\footnote{This follows from the countable iterability assumptions for $M$ and $N$ by relativizing Steel \cite[Lemma 3.1]{St95} to $\mathbb{R}$-premice and applying it to the images of $M$ and $N$ under the collapse embedding within the transitive collapse of a countable elementary substructure of some sufficiently large $V_\kappa$.}  It follows that for a closed, cofinal $C\subseteq\mathcal{P}_{\omega_1}(\mathbb{R})$, if $a \in C$, then every $\Pi^1_{n+1}$-iterable $a$-premouse which is sound and tame but not $n$-small and which projects to $a$ has an initial segment satisfying $\psi$. By a well-known characterization of closed-and-cofinalness (see e.g., Larson \cite{La04}), there is a function 
\[f:[\mathbb{R}]^{<\omega}\to \mathbb{R}\]
such that 
\[C = \{a \in \mathcal{P}_{\omega_1}(\mathbb{R}) : f[[a]^{<\omega}] \subseteq a\}.\]
Fix an enumeration $e$ of $\omega^{<\omega}$.
Let $\sigma$ be the following strategy for Player II in $G^{\text{odd}}_{\psi}$: Given a partial play $p$ in which reals $x_0, \hdots, x_k$ have been played,
\begin{enumerate}
\item if the finite set $x_0, \hdots, x_k$ is closed under $f$, let $\sigma(p) = x_0$,
\item otherwise, let $(i_0, i_1,\hdots i_l)$ be the least tuple (in the sense of $e$) such that $f(x_{i_0}, x_{i_1},\hdots, x_{i_l}) \not \in \{x_0, \hdots, x_k\}$ and set 
\[\sigma(p) = f(x_{i_0}, x_{i_1},\hdots, x_{i_l}).\]
\end{enumerate}
If $p$ is a full play by $\sigma$ and $p^*$ is its real part, then clearly $p^*$ is closed under $f$, so $p^* \in C$; thus $\sigma$ is a winning strategy for Player II.

Suppose now that Player II has a winning strategy, say, $\sigma$.
Since $n$ is odd, $M^\sharp_n(\mathbb{R})$ is countably $\Pi^1_{n+1}$-iterable.
Consider the run $p$ of $G^{\text{odd}}_{\psi}$ in which Player II moves according to $\sigma$ and Player I moves according to the theory of $M^\sharp_n(\mathbb{R})$ and plays the reals of a countable elementary substructure $H$ of $M^\sharp_n(\mathbb{R})$ such that $\sigma \in H$. Because $H$ is closed under finite tuples, every real in $p$ belongs to it; this gives rise to a model $N_p$ which embeds into $H$ (and hence into $M^\sharp_n(\mathbb{R})$) elementarily. Thus, $N_p$ is $\Pi^1_{n+1}$-iterable, not $n$-small, and has no proper initial segment which is not $n$-small. Since $\sigma$ was a winning strategy for Player II, $N_p$ must have a proper (and hence $n$-small) initial segment satisfying $\psi$. By elementarity, $M^\sharp_n(\mathbb{R})$ does too. All initial segments of $M^\sharp_n(\mathbb{R})$ are countably iterable, so the proof of the lemma is finished.
\end{proof}

This completes the proof of Theorem \ref{TheoremSigma1}. To derive Theorem \ref{theoremPointclass}, we need to combine  Theorem \ref{TheoremSigma1} with Lemma \ref{LemmaMnExists} below.

\section{Applications}
\label{SectApplications}
In this section, we list several applications of Theorem \ref{TheoremSigma1} and its proof. We begin with a result on the correctness of $M_n(\mathbb{R})$ concerning the existence of winning strategies.
Martin and Steel's proof of Theorem \ref{theoremPointclass} in the case $n=0$ yields the additional consequence that 
\[(\Game^\mathbb{R}\Pi^1_1)^{L(\mathbb{R})} = \Game^\mathbb{R}\Pi^1_1.\]
Ours does not, although we can prove it directly, assuming more determinacy. We do not know of a proof that does not require $\ad$ to hold in $M_n(\mathbb{R})$.

Begin by observing that if $M_n^\sharp(\mathbb{R})$ exists and satisfies $\ad$, then $\Sigma_1^{M_n(\mathbb{R})}$ has the scale property, by Steel \cite[Theorem 1.1]{St08b}. If so, then by a well-known theorem of Moschovakis \cite{Mo71}, $\Sigma_1^{M_n(\mathbb{R})}$ has the uniformization property.

\begin{theorem}\label{TheoremCorrectnessGames}
Suppose that $n\in\mathbb{N}$, $M_n^\sharp(\mathbb{R})$ exists and satisfies $\ad$, and that $\PI^1_{n+1}$-games on $\mathbb{R}$ are determined. Then,
\begin{enumerate}
\item If $n$ is even, then $(\Game^\mathbb{R}\Pi^1_{n+1})^{M_n(\mathbb{R})} = \Game^\mathbb{R}\Pi^1_{n+1}$.
\item If $n$ is odd, then $(\Game^\mathbb{R}\Sigma^1_{n+1})^{M_n(\mathbb{R})} = \Game^\mathbb{R}\Sigma^1_{n+1}$.
\end{enumerate}
\end{theorem}
\begin{proof}[Proof Sketch]
We suppose for definiteness that $n$ is even; the other case is similar. As the existence of winning strategies for games on $\mathbb{R}$ is upwards absolute from $M_n(\mathbb{R})$, 
it suffices to show that every $\Pi^1_{n+1}$ game on $\mathbb{R}$ has a winning strategy in $\Game^\mathbb{R}\Pi^1_{n+1}$. This follows from the proof of Moschovakis' Third Periodicity Theorem (see Moschovakis \cite[6E.1]{Mo09}) adapted to games on $\mathbb{R}$ (see also \cite[Lemma 3]{Ag18b} for a few more details). The only difference is that in the proof of \cite[6E.1]{Mo09}, which deals with games on $\mathbb{N}$, one defines each step of the winning strategy by selecting the least move which is ``minimal'' in the sense of \cite[6E.1]{Mo09}. In our setting, where moves are elements of $\mathbb{R}$, the collection of minimal moves is $\Game^\mathbb{R}\Pi^1_{n+1}$, so we may use the uniformization property to select one (this is why we require $\ad$).
\end{proof}

{\begin{center}
{{---}{---}}
\end{center}
}

What remains now is to keep the promises made in the introduction.
The next two lemmas together imply Theorem \ref{theoremGamesOnR}:

\begin{lemma} \label{LemmaMnExists}
Let $n$ be odd and 
suppose that $\PI^1_{n+2}$ games on $\mathbb{R}$ are determined, then $M_n^\sharp(\mathbb{R})$ exists.
\end{lemma}
\begin{proof}
We apply Lemma \ref{LemmaGphiLP} to $n+1$. It suffices to show, assuming $\PI^1_{n+2}$-determinacy for games on reals, that Player I has a winning strategy in $G^{\text{even}}_\varphi$, where $\varphi$ is the formula ``there is an active initial segment of me with $n$ Woodin cardinals.'' By $\PI^1_{n+2}$-determinacy for games on reals, it suffices to show that Player II does not have a winning strategy. Towards a contradiction, let $\sigma$ be a winning strategy for Player II. By $\PI^1_{n+2}$-determinacy for games on $\mathbb{N}$, $M_{n+1}^\sharp(A)$ exists for all $A\in\mathcal{P}_{\omega_1}(\mathbb{R})$. Moreover, each $M_{n+1}^\sharp(A)$ satisfies $\varphi$. Let $X$ be a countable elementary substructure of some large $V_\kappa$, with $\sigma \in X$, and let $W$ be the transitive collapse of $X$. Write $\bar\sigma$ for the image of $\sigma$ under the collapse embedding. Consider the play $p$ of $G^{\text{even}}_\varphi$ in which Player I plays the theory of the unique initial segment of $M_{n+1}(\mathbb{R}^W)$ which is a minimal $\varphi$-witness, together with $\mathbb{R}^W$, and Player II plays by $\sigma$. An argument as in \cite[Lemma 3.8]{AgMu} shows that this play is won by Player I, giving the desired contradiction.
\end{proof}

\begin{lemma}
Suppose that $M^\sharp_{n+1}(\mathbb{R})$ exists. Then, $\PI^1_{n+1}$ games on $\mathbb{R}$ are determined.
\end{lemma}
\proof
This is almost immediate from Martin and Steel's \cite{MaSt89} or Neeman's \cite{Ne10} proof of projective determinacy.

Work in $M^\sharp_{n+1}(\mathbb{R})$. We actually only need that there are $n$ Woodin cardinals and a measurable above (within this model). By \cite[Corollary 5.31]{Ne10}, $\PI^1_{n+1}$ games are determined. However, \cite[Corollary 5.31]{Ne10} is a corollary of \cite[Exercise 5.29]{Ne10} and \cite[Corollary 4.19]{Ne10}, which are stated for games on an arbitrary set $X$. If we take $X = \mathbb{R}$, we may conclude the determinacy of games on $\mathbb{R}$ which are $\PI^1_{n+1}$ in the sense of Remark \ref{RemarkProjective}. By Corollary \ref{CorollaryTopologies}, this suffices. 

We have shown that 
\[M^\sharp_{n+1}(\mathbb{R})\models \text{``$\PI^1_{n+1}$-determinacy for games on $\mathbb{R}$.''}\]
Since $M^\sharp_{n+1}(\mathbb{R})$ contains all reals, and thus all countable sequences of reals and all possible plays of any game on $\mathbb{R}$, the existence of winning strategies is absolute to $V$. This completes the proof.
\endproof

We conclude with the following strengthening of Theorem \ref{theoremLongGames}:

\begin{theorem}
Suppose $M^\sharp_n(\mathbb{R})$ exists.
The following are equivalent:
\begin{enumerate}
\item $\PI^1_{n+1}$-determinacy for games on $\mathbb{N}$ of length $\omega^2$;
\item $\PI^1_{n+1}$-determinacy for games on $\mathbb{R}$ of length $\omega$ and $M^\sharp_n(\mathbb{R})\models\ad$.
\end{enumerate}
\end{theorem}
\begin{proof}[Proof Sketch]
Suppose that $\PI^1_{n+1}$-determinacy for games on $\mathbb{N}$ of length $\omega^2$ holds and, towards a contradiction, that $M^\sharp_n(\mathbb{R})\not\models\ad$. By taking transitive collapses of countable elementary substructures, one sees that there is a closed, cofinal $C\subseteq\mathcal{P}_{\omega_1}(\mathbb{R})$ such that for all $A\in C$,
\[M_n^\sharp(A)\not\models\ad,\]
which contradicts \cite[Theorem 3.1]{AgMu}. 

Conversely, suppose that $\PI^1_{n+1}$ games on $\mathbb{R}$ of length $\omega$ are determined and that $M^\sharp_n(\mathbb{R})\models\ad$. Assume for definiteness that $n$ is even; the other case is similar.
By Theorem \ref{TheoremSigma1}, all sets in $\Game^\mathbb{R}\PI^1_{n+1}$ are determined. The result now follows by adapting the argument of \cite[Lemma 3.11]{Ag18c} exactly as in \cite[Theorem 3.12]{Ag18c}. This is possible because all properties about $L(\mathbb{R})$ used in argument have been verified above to also hold of $M^\sharp_n(\mathbb{R})$; namely, the fact that $L(\mathbb{R})$ is correct about Player I winning $\PI^1_1$ games on $\mathbb{R}$
(which generalizes to $M^\sharp_n(\mathbb{R})$ by Theorem \ref{TheoremCorrectnessGames}), the fact that $\Sigma_1^{L(\mathbb{R})}$ has the uniformization property (which generalizes to 
$M^\sharp_n(\mathbb{R})$ by the remark in the beginning of this section), as well as Solovay's basis theorem (which can be proved for $M^\sharp_n(\mathbb{R})$ the same way as for $L(\mathbb{R})$; see Koellner-Woodin \cite{KW10} for a proof).
\end{proof}
\endproof

\bibliographystyle{abbrv}
\bibliography{References}

\end{document}